\documentclass{amsart}%
\usepackage{graphicx}
\usepackage{amscd}
\usepackage{amsmath}
\usepackage{amsfonts}
\usepackage{amssymb}%
\newtheorem{theorem}{Theorem}[section]

\newtheorem{lemma}[theorem]{Lemma}

\numberwithin{equation}{section}

\begin{document}
\title[ ]{Erratum to \textquotedblleft Cobordisms of maps with singularities of given
class\textquotedblright}
\author{YOSHIFUMI ANDO}
\address{Department of Mathematical Sciences, Faculty of Science, Yamaguchi University,
Yamaguchi, 753-8512, Japan}
\email{andoy@yamaguchi-u.ac.jp}
\thanks{$2000$ \textit{Mathematics Subject Classification.} $57R45, 57R90, 58A20$}

\begin{abstract}
We correct the proof in the unoriented case of Theorem 1.2 in the paper
\textquotedblleft Cobordisms of maps with singularities of given
class\textquotedblright.

\end{abstract}
\maketitle

Andras Sz\"{u}cs pointed out that the proof of Theorem 1.2 in the unoriented
case is false. This is caused by a wrong assertion that any element of
$\mathfrak{N}(n,P;\boldsymbol{\Omega,\Omega}_{\mathcal{K}})$ is of order two.
For example, it has been proved in \cite{wells} that there can be elements of
infinite order in the cobordism groups of immersions.

Let $X$ be a connected space with a simplicial complex structure. A path of a
point $x$ to a point $y$ in $X$ is denoted such as $\alpha_{xy}$\ or
$\beta_{xy}$. The inverse of a path $\alpha_{xy}$\ is denoted by $\alpha
_{xy}^{-1}$. We take a base point $e$\ of $X$. Let $w:H_{1}%
(X;\mathbb{Z)\rightarrow\{\pm}1\mathbb{\}}$\ denote an epimorphism. The
homology class of a loop $\alpha_{xx}$\ in $H_{1}(X;\mathbb{Z)}$\ is denoted
by $[\alpha_{xx}]$. Let $p:\widetilde{X}\rightarrow X$\ denote a double
covering associated to $w$. The space $\widetilde{X}$\ is defined to be a
space consisting of all pairs $(x,\alpha_{xe})$\ under the equivalence
relation such that $(x,\alpha_{xe})$ and $(x,\beta_{xe})$\ are equivalent if
and only if $w([\alpha_{xe}^{-1}\beta_{xe}])=1$. Let $\alpha_{xe}^{\prime}%
$\ denote a path such that $w([\alpha_{xe}^{-1}\alpha_{xe}^{\prime}])=-1$.
Then $p^{-1}(x)$\ consists of $(x,\alpha_{xe})$\ and $(x,\alpha_{xe}^{\prime
})$. Let us recall the local coefficients $\mathcal{Z}=\{\mathbb{Z}_{x}%
\}$\ over $X$\ associated to the covering $p:\widetilde{X}\rightarrow X$
introduced in \cite{ste}, where we have\ an isomorphism $(\alpha_{xe}%
)_{\#}:\mathbb{Z}_{x}\rightarrow\mathbb{Z}_{e}$\ with required properties. Let
us consider the local coefficients $\mathcal{F}=\{H_{0}(p^{-1}(x);\mathbb{Z}%
)\}$ over $X$ Here, an element of $H_{0}(p^{-1}(x);\mathbb{Z})$\ is expressed
as $\{s,t\}$, where $s\in H_{0}((x,\alpha_{xe});\mathbb{Z})$, $t\in
H_{0}((x,\alpha_{xe}^{\prime});\mathbb{Z})$, and these two $0$-th homology
groups are canonicaly isomorphic to $\mathbb{Z=Z}_{e}$. However, we cannot
give a specified order to these two elements. We note that there exists the
isomorphisms%
\[
(\alpha_{xe})_{\#}:\mathbb{Z}_{x}\rightarrow\mathbb{Z}_{e}\text{ \ \ \ and
\ \ \ }(\alpha_{xe}^{\prime})_{\#}:\mathbb{Z}_{x}\rightarrow\mathbb{Z}_{e}.
\]
Let $\mathfrak{i}:X\times\mathbb{Z\rightarrow}\mathcal{F}$\ denote a
monomorphism defined by $\mathfrak{i}(x,s)=\{s,s\}$, which lies in
$H_{0}(p^{-1}(x);\mathbb{Z})$. Let $\mathfrak{j:}\mathcal{F}\rightarrow
\mathcal{Z}$\ denote an epimorphism defined by $\mathfrak{j}(\{s,t\})=(\alpha
_{xe})_{\#}^{-1}(s)+(\alpha_{xe}^{\prime})_{\#}^{-1}(t)\in\mathbb{Z}_{x}$.
This is well defined. Then it is not difficult to see that the sequence of the
local coefficients%
\[
0\rightarrow X\times\mathbb{Z}\rightarrow\mathcal{F}\rightarrow\mathcal{Z}%
\rightarrow0
\]
is exact. Indeed, $(\alpha_{xe})_{\#}^{-1}(s)+(\alpha_{xe}^{\prime})_{\#}%
^{-1}(t)=0$ if and only if $s=t$. Therefore, we have the exact sequence%
\begin{equation}
\rightarrow H_{i+1}(X,Y;\mathcal{Z}\mathbb{)}\overset{\partial}{\rightarrow
}H_{i}(X,Y;X\times\mathbb{Z})\rightarrow H_{i}(X,Y;\mathcal{F})\rightarrow
H_{i}(X,Y;\mathcal{Z}\mathbb{)}\overset{\partial}{\rightarrow}%
\end{equation}
for a pair $(X,Y)$, where $Y$ is a subcomplex.

The following lemma may be well-known, though we give an elementary proof for completeness.

\begin{lemma}
Let $(X,Y)$\ be as above and $\widetilde{Y}=p^{-1}(Y)$. Then $H_{\ast
}(\widetilde{X},\widetilde{Y};\mathbb{Z)}$ is canonically\ isomorphic
to\ $H_{\ast}(X,Y;\mathcal{F)}$.
\end{lemma}

\begin{proof}
We provide $\widetilde{X}$ and $\widetilde{Y}$\ with simplicial complex
structures induced from $X$ and $Y$.\ A $j$-simplex $|p_{0}\cdots p_{j}|$
of\ $X$\ constitutes the two simplexes in $p^{-1}(|p_{0}\cdots p_{j}|)$, which
we denote by $(|p_{0}\cdots p_{j}|,\alpha_{p_{0}e})$ and $(|p_{0}\cdots
p_{j}|,\alpha_{p_{0}e}^{\prime})$. This notation will be reasonable. Let $J$
denote the set consisting of all $j$-simplexes in $X$. Let us define a chain
isomorphism%
\[
\phi_{j}:C_{j}(X;\mathcal{F})\rightarrow C_{j}(\widetilde{X};\mathbb{Z)}%
\]
as follows. Let $s\in H_{0}((x,\alpha_{p_{0}e});\mathbb{Z})$\ and $t\in
H_{0}((x,\alpha_{p_{0}e}^{\prime});\mathbb{Z})$.\ An element of $C_{j}%
(X;\mathcal{F})$\ is written as
\[
\sum_{|p_{0}\cdots p_{j}|\in J}(s+t)|p_{0}\cdots p_{j}|,
\]
where $s+t$\ is zero except for finite $j$-simplexes. Let $\phi_{j}$ map it
onto the following element of $C_{j}(\widetilde{X};\mathbb{Z)}$.%
\[
\sum_{|p_{0}\cdots p_{j}|\in J}s(|p_{0}\cdots p_{j}|,\alpha_{p_{0}e}%
)+t(|p_{0}\cdots p_{j}|,\alpha_{p_{0}e}^{\prime}).
\]
It is not difficult to see that $\phi_{j}$ is an isomorphism.

Next we show that $\phi_{j}$ is compatible with the boundary operators. Let
$\overline{p_{0}p_{1}}$ be a segment in $|p_{0}\cdots p_{j}|$ and let
$(\overline{p_{0}p_{1}})_{\#}:H_{0}(p^{-1}(p_{0});\mathbb{Z})\rightarrow
H_{0}(p^{-1}(p_{1});\mathbb{Z})$ be the canonical isomorohism. Then we have%
\begin{align*}
&  \phi_{j-1}\circ\partial((s+t)|p_{0}\cdots p_{j}|)\\
&  =\phi_{j-1}\{(\overline{p_{0}p_{1}})_{\#}(s+t)|p_{1}\cdots p_{j}%
|+\sum_{m=1}^{j}(-1)^{m}(s+t)|p_{0}\cdots p_{m-1}p_{m+1}\cdots p_{j}|\}\\
&  =((\overline{p_{0}p_{1}})_{\#}(s))((|p_{1}\cdots p_{j}|,\alpha_{p_{1}%
e})+((\overline{p_{0}p_{1}})_{\#}(t))((|p_{1}\cdots p_{j}|,\alpha_{p_{1}%
e}^{\prime})\}\\
&  \text{ \ \ \ \ \ \ \ \ \ \ }+\sum_{m=1}^{j}(-1)^{m}\{s(|p_{0}\cdots
p_{m-1}p_{m+1}\cdots p_{j}|,\alpha_{p_{0}e})\\
&  \text{
\ \ \ \ \ \ \ \ \ \ \ \ \ \ \ \ \ \ \ \ \ \ \ \ \ \ \ \ \ \ \ \ \ \ \ \ \ }%
+t(|p_{0}\cdots p_{m-1}p_{m+1}\cdots p_{j}|,\alpha_{p_{0}e}^{\prime})\}\\
&  =\partial\left\{  \sum_{|p_{0}\cdots p_{j}|\in J}s(|p_{0}\cdots
p_{j}|,\alpha_{p_{0}e})+t(|p_{0}\cdots p_{j}|,\alpha_{p_{0}e}^{\prime
})\right\}  \\
&  =\partial\circ\phi_{j}((s+t)|p_{0}\cdots p_{j}|)
\end{align*}
This proves the assertion.
\end{proof}

Now we are ready to prove Theorem 1.2 in the unoriented case.

\begin{proof}
[Proof of Theorem 1.2 in the unoriented case.]Let $\ell\gg n,p$. Let
$\widehat{\gamma}_{\boldsymbol{\Omega}}^{\ell}$ and $\widehat{\gamma
}_{\boldsymbol{\Omega}_{\mathcal{K}}^{1}}^{\ell}$ be the vector bundles over
the spaces $\boldsymbol{\Omega}$ and $\boldsymbol{\Omega}_{\mathcal{K}}^{1}$
given in \cite[Introduction]{ando} and let $(\widehat{\gamma}%
_{\boldsymbol{\Omega}}^{\ell})_{x}$ and $(\widehat{\gamma}_{\boldsymbol{\Omega
}_{\mathcal{K}}^{1}}^{\ell})_{x}$ be their fibers over a point $x$
respectively. Let
\begin{align*}
\mathcal{Z}(w^{\boldsymbol{\Omega}})  &  =\left\{  H_{\ell}\left(  \left(
\widehat{\gamma}_{\boldsymbol{\Omega}}^{\ell}\right)  _{x},\left(
\widehat{\gamma}_{\boldsymbol{\Omega}}^{\ell}\right)  _{x}\backslash\left\{
0\right\}  ;\mathbb{Z}\right)  \right\}  _{x\in\boldsymbol{\Omega}},\\
\mathcal{Z}(w^{\boldsymbol{\Omega}_{\mathcal{K}}^{1}})  &  =\left\{  H_{\ell
}\left(  \left(  \widehat{\gamma}_{\boldsymbol{\Omega}_{\mathcal{K}}^{1}%
}^{\ell}\right)  _{x},\left(  \widehat{\gamma}_{\boldsymbol{\Omega
}_{\mathcal{K}}^{1}}^{\ell}\right)  _{x}\backslash\left\{  0\right\}
;\mathbb{Z}\right)  \right\}  _{x\in\boldsymbol{\Omega}_{\mathcal{K}}^{1}}%
\end{align*}
denote the local coefficient systems associated to the first Stiefel-Whitney
classes of the grassmann manifolds $G_{n,\ell}$ and $G_{n,\ell+1}$ in
$\boldsymbol{\Omega}$ and $\boldsymbol{\Omega}_{\mathcal{K}}^{1}$
respectively. It follows from \cite[Theorem 1.2]{thom} that
\begin{align*}
H_{i+\ell}\left(  T\left(  \widehat{\gamma}_{\boldsymbol{\Omega}}^{\ell
}\right)  ;\mathbb{Z}\right)   &  \approx H_{i}(\boldsymbol{\Omega
};\mathcal{Z}(w^{\boldsymbol{\Omega}})),\\
H_{i+\ell}\left(  T\left(  \widehat{\gamma}_{\boldsymbol{\Omega}_{\mathcal{K}%
}^{1}}^{\ell}\right)  ;\mathbb{Z}\right)   &  \approx H_{i}(\boldsymbol{\Omega
}_{\mathcal{K}}^{1};\mathcal{Z}(w^{\boldsymbol{\Omega}_{\mathcal{K}}^{1}})).
\end{align*}
Since $T(\widehat{\gamma}_{\boldsymbol{\Omega}}^{\ell})$ and $T(\widehat
{\gamma}_{\boldsymbol{\Omega}_{\mathcal{K}}^{1}}^{\ell})$ are simply
connected, it is sufficient for Theorem 1.2 to prove that
\[
(\Delta^{(\Omega,\Omega_{\mathcal{K}})})_{\ast}:H_{i}(\boldsymbol{\Omega
};\mathcal{Z}(w^{\boldsymbol{\Omega}}))\rightarrow H_{i}(\boldsymbol{\Omega
}_{\mathcal{K}}^{1};\mathcal{Z}(w^{\boldsymbol{\Omega}_{\mathcal{K}}^{1}}))
\]
is an isomorphism for $0\leqq i<n$ and an epimorphism for $i=n$.

By applying the mapping cylinder of $\Delta^{(\Omega,\Omega_{\mathcal{K}})}$,
we may suppose that $\boldsymbol{\Omega}$ is a subspace of $\boldsymbol{\Omega
}_{\mathcal{K}}^{1}$, $\Delta^{(\Omega,\Omega_{\mathcal{K}})}$\ is an
inclusion, and that $(\Delta^{(\Omega,\Omega_{\mathcal{K}})})^{\ast
}(\mathcal{Z}(w^{\boldsymbol{\Omega}_{\mathcal{K}}^{1}}))$ is isomorphic to
$\mathcal{Z}(w^{\boldsymbol{\Omega}})$. By considering relative homology
groups the above assertion for $(\Delta^{(\Omega,\Omega_{\mathcal{K}})}%
)_{\ast}$ is equivalent to the assertion that
\begin{equation}
H_{i}(\boldsymbol{\Omega}_{\mathcal{K}}^{1},\boldsymbol{\Omega};\mathcal{Z}%
(w^{\boldsymbol{\Omega}_{\mathcal{K}}^{1}}))
\end{equation}
vanishes for $0\leqq i\leqq n$.

Recall $\pi_{1}(G_{n,\ell})=\pi_{1}(G_{n,\ell+1})=\{\pm1\}$. Let
$p_{\boldsymbol{\Omega}_{\mathcal{K}}^{1}}:(\widetilde{\boldsymbol{\Omega}%
}_{\mathcal{K}}^{1},\widetilde{\boldsymbol{\Omega}})\rightarrow
(\boldsymbol{\Omega}_{\mathcal{K}}^{1},\boldsymbol{\Omega)}$ denote the double
covering of $\boldsymbol{\Omega}$ and $\boldsymbol{\Omega}_{\mathcal{K}}^{1}$
associated to the first Stiefel-Whitney classes of the grassmann manifolds
$G_{n,\ell}$ and $G_{n,\ell+1}$ respectively. Let $\mathcal{F}%
(\boldsymbol{\Omega}_{\mathcal{K}}^{1})$\ denote the local coefficients
$\{H_{0}((p_{\boldsymbol{\Omega}_{\mathcal{K}}^{1}})^{-1}(x);\mathbb{Z})\}$ as
above. It is enough for the assertion concerning (0.2) to prove that
$H_{i}(\boldsymbol{\Omega}_{\mathcal{K}}^{1},\boldsymbol{\Omega};\mathbb{Z})$
and $H_{i}(\boldsymbol{\Omega}_{\mathcal{K}}^{1},\boldsymbol{\Omega
};\mathcal{F}(\boldsymbol{\Omega}_{\mathcal{K}}^{1}))$ vanish for $0\leqq
i\leqq n$ by the exact sequence in (0.1).

On the otherhand, we have by Lemma 0.1 that
\[
H_{i}(\boldsymbol{\Omega}_{\mathcal{K}}^{1},\boldsymbol{\Omega};\mathcal{F}%
(\boldsymbol{\Omega}_{\mathcal{K}}^{1}))\approx H_{i}(\widetilde
{\boldsymbol{\Omega}}_{\mathcal{K}}^{1},\widetilde{\boldsymbol{\Omega}%
};\mathbb{Z}).
\]
Hence, we show that $H_{i}(\boldsymbol{\Omega}_{\mathcal{K}}^{1}%
,\boldsymbol{\Omega};\mathbb{Z})$ and $H_{i}(\widetilde{\boldsymbol{\Omega}%
}_{\mathcal{K}}^{1},\widetilde{\boldsymbol{\Omega}};\mathbb{Z})$ vanish for
$0\leqq i\leqq n$. In fact, it follows from \cite[Lemma7.2]{ando} that
$(\Delta^{(\Omega,\Omega_{\mathcal{K}})})_{\ast}:\pi_{i}(\boldsymbol{\Omega
})\rightarrow\pi_{i}(\boldsymbol{\Omega}_{\mathcal{K}}^{1})$ is an isomorphism
for $0\leqq i<n$ and an epimorphism for $i=n$. Namely, both of $\pi
_{i}(\boldsymbol{\Omega}_{\mathcal{K}}^{1},\boldsymbol{\Omega})$ and $\pi
_{i}(\widetilde{\boldsymbol{\Omega}}_{\mathcal{K}}^{1},\widetilde
{\boldsymbol{\Omega}})$ vanish for $0\leqq i\leqq n$, and so do $H_{i}%
(\boldsymbol{\Omega}_{\mathcal{K}}^{1},\boldsymbol{\Omega};\mathbb{Z})$ and
$H_{i}(\widetilde{\boldsymbol{\Omega}}_{\mathcal{K}}^{1},\widetilde
{\boldsymbol{\Omega}};\mathbb{Z})$. This proves Theorem 1.2.
\end{proof}

In the proof of \cite[Proposition 10.3]{ando} the phrase \textquotedblleft the
similar argument as in the proof of Theorem 1.2\textquotedblright\ should be
read in the context of this note.

The author would like to express his sincere gratitude to Professors Andras Sz\"{u}cs.

\end{document}